\documentclass[a4paper,12pt]{article}

\usepackage[left=2cm,right=2cm, top=2cm,bottom=3cm,bindingoffset=0cm]{geometry}

\usepackage{verbatim}
\usepackage{amsmath}
\usepackage{amsthm}
\usepackage{amssymb}
\usepackage{delarray}
\usepackage{cite}
\usepackage{hyperref}
\usepackage{mathrsfs}
\usepackage{tikz}
\usetikzlibrary{patterns}
\usepackage{caption}
\DeclareCaptionLabelSeparator{dot}{. }
\newcommand{\al}{\alpha}

\newcommand{\ga}{\gamma}
\newcommand{\de}{\delta}
\newcommand{\la}{\lambda}

\newcommand{\vv}{\varphi}
\newcommand{\iy}{\infty}

\theoremstyle{plain}

\numberwithin{equation}{section}

\newtheorem{thm}{Theorem}[section]
\newtheorem{lem}[thm]{Lemma}
\newtheorem{prop}[thm]{Proposition}

\theoremstyle{definition}

\newtheorem{alg}[thm]{Algorithm}
\newtheorem{ip}[thm]{Inverse Problem}

\theoremstyle{remark}

\DeclareMathOperator*{\Res}{Res}
\DeclareMathOperator{\rank}{rank}

\DeclareMathOperator{\Ran}{Ran}
\DeclareMathOperator{\Ker}{Ker}

 \allowdisplaybreaks

\begin{document}

\begin{center}
{\large\bf Inverse problem solution and spectral data characterization for the matrix Sturm-Liouville operator with singular potential}
\\[0.2cm]
{\bf Natalia P. Bondarenko} \\[0.2cm]
\end{center}

\vspace{0.5cm}

{\bf Abstract.} The matrix Sturm-Liouville operator on a finite interval with singular potential of class $W_2^{-1}$ and the general self-adjoint boundary conditions is studied. This operator generalizes the Sturm-Liouville operators on geometrical graphs. We investigate the inverse problem that consists in recovering the considered operator from the spectral data (eigenvalues and weight matrices). The inverse problem is reduced to a linear equation in a suitable Banach space, and a constructive algorithm for the inverse problem solution is developed. Moreover, we obtain the spectral data characterization for the studied operator.

\medskip


\medskip


\vspace{1cm}

\section{Introduction}

This paper is devoted to an inverse spectral problem for the matrix Sturm-Liouville operator $-Y'' + Q(x) Y$, where $Q(x)$ is an $(m \times m)$-matrix function called the \textit{potential}.

Inverse problems of spectral analysis consist in reconstruction of operators from their spectral information. The greatest success in inverse problem theory has been achieved for the \textit{scalar} Sturm-Liouville operators (for $m = 1$), see the classical monographs \cite{Mar77, Lev84, PT87, FY01} and references therein. \textit{Matrix} Sturm-Liouville operators have been intensively studied in connection with various applications. In particular, inverse problems for such operators are used in quantum mechanics \cite{AM63}, in elasticity theory \cite{BHN95}, for description of electromagnetic waves \cite{BS95} and nuclear structure \cite{Chab04}, for solving matrix nonlinear evolution equations by inverse spectral transform \cite{CD77}.

For the matrix Sturm-Liouville operators on a \textit{finite interval}, the majority of studies deal with the Dirichlet boundary conditions
$$
Y(0) = Y(\pi) = 0
$$
or the Robin boundary conditions
$$
Y'(0) - H_1 Y(0) = 0, \quad Y'(\pi) + H_2 Y(\pi) = 0,
$$
where $H_1$ and $H_2$ are constant $(m \times m)$-matrices.
Uniqueness of recovering such operators from various spectral characteristics has been proved by Carlson \cite{Car02}, Chabanov \cite{Chab04}, Malamud \cite{Mal05}, Yurko \cite{Yur06-uniq}, and Shieh \cite{Sh07}. Yurko \cite{Yur06} proposed a constructive method, based on spectral mappings, for solving such inverse problems. Further, this method has been developed by Bondarenko \cite{Bond11} for working with multiple eigenvalues. The most difficult and, at the same time, the most important issue of inverse problem theory is the spectral data characterization. For the matrix Sturm-Liouville operators on a finite interval, this issue has been independently solved by Chelkak and Korotyaev \cite{CK09}, by Mykytyuk and Trush \cite{MT09}, and by Bondarenko \cite{Bond11, Bond12}. 
The latter approach was also generalized for a certain class of non-self-adjoint matrix Sturm-Liouville operators \cite{Bond19}.

The present paper deals with the matrix Sturm-Liouville operator with the self-adjoint boundary conditions in the general form defined below.
Denote by $\mathbb C^m$ and $\mathbb C^{m \times m}$ the spaces of complex $m$-vectors and $(m \times m)$-matrices, respectively. For an interval $\mathbb I$ and a class $\mathcal A(\mathbb I)$ of functions defined on $\mathbb I$ (e.g., $\mathcal A = L_2, C, \dots$), we denote by $\mathcal A(\mathbb I; \mathbb C^m)$ and $\mathcal A(\mathbb I; \mathbb C^{m \times m})$ the classes of complex-valued $m$-vector functions and $(m \times m)$-matrix functions, respectively, with entries from $\mathcal A(\mathbb I)$.

Consider the matrix Sturm-Liouville problem $L = L(\sigma, T_1, T_2, H_2)$:
\begin{gather} \label{eqv}
    \ell Y := -(Y^{[1]})' - \sigma(x) Y^{[1]} - \sigma^2(x) Y = \la Y, \quad x \in (0, \pi), \\ \label{bc}
    V_1(Y) := T_1 Y^{[1]}(0) - T_1^{\perp} Y(0) = 0, \quad V_2(Y) := T_2 (Y^{[1]}(\pi) - H_2 Y(\pi)) - T_2^{\perp} Y(\pi) = 0. 
\end{gather}
where $Y = [y_j(x)]_{j = 1}^m$ is a vector function, $\sigma \in L_2((0, \pi); \mathbb C^{m \times m})$, $\sigma(x) = (\sigma(x))^{\dagger}$ a.e. on $(0, \pi)$, $Y^{[1]}(x) := Y'(x) - \sigma(x) Y(x)$ is the {\it quasi-derivative}, $\la$ is the spectral parameter, for $j = 1, 2$, $T_j \in \mathbb C^{m \times m}$, $T_j$ is an orthogonal projection matrix, $T_j^{\perp} = I - T_j$, $H_2 \in \mathbb C^{m \times m}$, $H_2 = H_2^{\dagger} = T_2 H_2 T_2$, $I$ is the $(m \times m)$-unit matrix, the symbol $\dagger$ denotes the conjugate transform.
Under these assumptions, the problem $L$ is self-adjoint.
We suppose that $Y$ belongs to the domain
$$
\mathcal D(L) := \{ Y \colon Y, Y^{[1]} \in AC([0, \pi]; \mathbb C^m), (Y^{[1]})' \in L_2((0, \pi); \mathbb C^m) \}.
$$
Equation~\eqref{eqv} can be rewritten in the equivalent form
$$
    -Y'' + Q(x) Y = \la Y, \quad x \in (0, \pi),
$$
with the singular potential $Q(x) = \sigma'(x)$ of class $W_2^{-1}((0, \pi); \mathbb C^{m \times m})$. The derivative of $L_2$-function is understood in the sense of distributions. However, it is more convenient to use the form~\eqref{eqv}.

Relations~\eqref{bc} describe the general self-adjoint form of separated boundary conditions. The matrix Sturm-Liouville operator given by~\eqref{eqv}-\eqref{bc} causes interest because it generalizes Sturm-Liouville operators on geometrical graphs. The latter operators are used for modeling wave propagation in graph-like structures consisting of thin tubes, strings, beams, etc. 
Differential operators on graphs attract much attention of mathematicians and physicists in recent years in connection with applications in nanotechnology, organic chemistry, mechanics, and other branches of science and engineering (see \cite{Kuch02, PPP04, Exner08, BK13} and references therein). The general self-adjoint boundary conditions in the form
$$
T Y'(v) + H Y(v) = 0, \quad T^{\perp} Y(v) = 0,
$$
where $T$ and $T^{\perp}$ are complimentary projection matrices, $H = H^{\dagger} = T H T$, have been introduced by Kuchment \cite{Kuch04}. In the literature (see, e.g., \cite{Now07}), the other equivalent forms of parametrization also appear:
$$
A Y(v) + B Y'(v) = 0,
$$
where the $(m \times 2m)$-matrix $[A, B]$ has the maximal rank $m$ and the matrix $A B^{\dagger}$ is Hermitian, and
$$
- i (U + I) Y(v) + (U - I) Y'(v) = 0,
$$
where $U$ is a unitary matrix.

Inverse problems for the matrix Sturm-Liouville operator on a finite interval with general self-adjoint boundary conditions and \textit{regular} potential of class $L_2$ have been recently studied by Xu \cite{Xu19}. However, paper \cite{Xu19} is only concerned with uniqueness theorems. The issues of constructive solution and spectral data characterization for this operator appeared to be more difficult for investigation because of complex asymptotic behavior of the spectrum and structural properties of the problem. In \cite{Bond19-BVP, Bond20}, properties of the spectral data have been investigated for the matrix Sturm-Liouville operator with boundary condition in the general self-adjoint form at $x = \pi$ and with Dirichlet boundary condition at $x = 0$. Further, a constructive solution procedure has been developed for the corresponding inverse spectral problem (see \cite{Bond20-IPSE}). Those results have been applied for obtaining the spectral data characterization for the Sturm-Liouville operator on the star-shaped graph (see \cite{Bond20-nsc}).

In addition, it worth mentioning that inverse scattering problems have been studied for the matrix Sturm-Liouville operators on the half-line and on the line (see, e.g., \cite{AM63, CD77, Har02, Har05, Wad80, Olm85, AW21}). In particular, Harmer \cite{Har02, Har05} and Aktosun and Weder \cite{AW21} investigated inverse scattering on the half-line with boundary condition in the general self-adjoint form at the origin. Harmer \cite{Har02} also applied those results to the inverse scattering problem on the star-shaped graph consisting of infinite rays. However, the matrix Sturm-Liouville operators on infinite domains usually have a bounded set of eigenvalues, so the inverse problems for them are in some sense easier for investigation than analogous problems on a finite interval.

The majority of mentioned results deal with the case of \textit{regular} (square summable or summable) potentials. For the Sturm-Liouville operators with \textit{singular} (distributional) potentials, there is an extensive literature concerning the \textit{scalar} case. Inverse problems for the scalar operators in the form $-(y^{[1]})' - \sigma(x) y^{[1]} - \sigma^2(x) y$, $\sigma \in L_2(0, \pi)$, were studied by Hryniv and Mykytuyk \cite{HM-sd, HM-2sp}, Savchuk and Shkalikov \cite{SS05}, Djakov and Mityagin \cite{DM09}, and by other authors. Mykytyuk and Trush \cite{MT09} investigate inverse problems for the matrix Sturm-Liouville operators with potential of class $W_2^{-1}$ on a finite interval in a special form, which differs from \eqref{eqv} and can be easily reduced to a Dirac-type operator. Analogous reduction was applied by Eckhardt and co-authors \cite{Eckh14, Eckh15} to the matrix Sturm-Liouville operators on the half-line and on the line.

In this paper, we solve the inverse spectral problem for the matrix Sturm-Liouville operator \eqref{eqv}-\eqref{bc} with singular potential and with general self-adjoint boundary conditions at the both ends of the interval. We obtain an algorithm for reconstruction of the operator by its spectral data and provide the spectral data characterization. On the one hand, our approach is based on the spectral properties of the operator~\eqref{eqv}-\eqref{bc} obtained in our previous study~\cite{Bond-dir}. On the other hand, we rely on the method of spectral mappings for constructive solution of the inverse problem. This method has been initially developed by Yurko for operators with regular coefficients (see \cite{FY01}). This method allows one to reduce a nonlinear inverse problem to a linear equation in a suitable Banach space. Such reduction leads to a constructive procedure for solving an inverse problem and also can be used for investigating global solvability, local solvability, stability, and other issues of inverse problem theory. Yurko's method has been modified for the Sturm-Liouville operators with singular potentials by Freiling, Ignatiev, and Yurko \cite{FIY08} and by Bondarenko \cite{Bond-scal}. An approach to inverse problems for the matrix Sturm-Liouville operators has been developed in \cite{Bond11, Bond19, Bond20-IPSE, Bond20-nsc}. In the present paper, we combine the ideas of the mentioned studies to solve the inverse problem for the operator \eqref{eqv}-\eqref{bc}. 

The paper is organized as follows. In Section~2, we describe asymptotical and structural properties of the spectral data, formulate the inverse problem, the corresponding uniqueness theorem, and our main theorem (Theorem~\ref{thm:nsc}) on the characterization of the spectral data. The proof of Theorem~\ref{thm:nsc} is contained in Sections~3-6. In Section~3, we reduce the nonlinear inverse problem to a linear equation in a special Banach space. That equation is called {\it the main equation} of the inverse problem. In Section~4, we obtain auxiliary estimates concerning the operator participating in the main equation and some other characteristics. In Section~5, it is proved that, under the conditions of Theorem~\ref{thm:nsc}, the main equation is uniquely solvable. In Section~6, the proof of Theorem~\ref{thm:nsc} is finished. By using the solution of the main equation, we construct $\sigma$ and $H_2$, and finally arrive at Algorithm~\ref{alg:1} for constructive solution of the inverse problem. 

We overcome the following difficulties specific for our problem.

\begin{enumerate}
    \item The problem $L$ can have an infinite number of groups of multiple and/or asymptotically multiple eigenvalues. Therefore, in the construction of the main equation in Section~3, we use the special grouping \eqref{defG} of the eigenvalues with respect to their asymptotics.
    
    \item Because of the singular potential, we need to obtain some precise estimates related with the operator participating in the main equation (see Lemmas~\ref{lem:Rl2}-\ref{lem:Rl1}). These estimates play an important role in the proofs of the main equation solvability. 
    Such estimates do not needed in the case of regular potential.
    
    \item When the matrix function $\sigma(x)$ is constructed by using the spectral data, we cannot directly substitute this function into equation~\eqref{eqv} and so have to approximate it by smooth matrix functions $\sigma^N$.
\end{enumerate}

\section{Preliminaries and main results}

In this section, we define the spectral data and provide their properties obtained in \cite{Bond-dir}. Further, we formulate the inverse problem (Inverse Problem~\ref{ip:1}), the corresponding uniqueness theorem (Proposition~\ref{prop:uniq}), and our main result (Theorem~\ref{thm:nsc}). The latter theorem gives necessary and sufficient conditions for the inverse problem solvability, or, in other words, the spectral data characterization.

Let us start with the {\bf notations}:

\begin{enumerate}
\item Denote $\rho := \sqrt \la$, $\arg \rho \in \left[ -\frac{\pi}{2}, \frac{\pi}{2} \right)$ (unless stated otherwise).
\item We use the Euclidean norm in $\mathbb C^m$:
$$
  \| a \| = \left( \sum_{j = 1}^m |a_j|^2 \right)^{1/2}, \quad a = [a_j]_{j = 1}^m,    
$$
and the corresponding matrix norm  $\| A \|$ equal to the maximal singular value of $A$.
\item The scalar product in the Hilbert space
$L_2(\mathbb I; \mathbb C^m)$ is defined as follows:
\begin{gather*}
(Y, Z) = \int_{\mathbb I} (Y(x))^{\dagger} Z(x) \, dx = \sum_{j = 1}^m \int_{\mathbb I} \overline{y_j(x)} z_j(x) \, dx, \\ Y = [y_j(x)]_{j = 1}^m, Z = [z_j(x)]_{j = 1}^m \in L_2(\mathbb I; \mathbb C^m).
\end{gather*}
\item The same symbol $C$ is used for various positive constants independent of $n$, $x$, $\la$, etc.
\end{enumerate}

Let $\vv(x, \la)$ be the matrix solution of equation~\eqref{eqv} satisfying the initial conditions $\vv(0, \la) = T_1$, $\vv^{[1]}(0, \la) = T_1^{\perp}$. Clearly, the matrix functions $\vv(x, \la)$ and $\vv^{[1]}(x, \la)$ are entire in $\la$ for each fixed $x \in [0, \pi]$. The eigenvalues of the problem $L$ coincide with the zeros of the entire characteristic function $\Delta(\la) := V_2(\vv(x, \la))$ with their multiplicities.

The matrix function $\vv(x, \la)$ can be represented in the form
\begin{equation} \label{asymptvv}
\vv(x, \la) = (\cos \rho x \, T_1 + \sin \rho x \, T_1^{\perp} + K_x(\rho)) (T_1 + \rho^{-1} T_1^{\perp}),
\end{equation}
where
$$
K_x(\rho) = \int_{-x}^x \mathscr K(x, t) \exp(i \rho t) \, dt,
$$
the kernel $\mathscr K(x, .)$ belongs to $L_2((-x, x); \mathbb C^{m \times m})$ for each fixed $x \in [0, \pi]$ and the norm $\| K(x, .) \|_{L_2((-x, x); \mathbb C^{m \times m})}$ is bounded uniformly by $x \in (0, \pi]$. Using~\eqref{asymptvv} and the analogous relation for $\vv^{[1]}(x, \la)$, we have proved the following proposition in \cite{Bond-dir}.

\begin{prop} \label{prop:asymptla}
The spectrum of $L$ is a countable set of real eigenvalues $\{ \la_{nk} \}_{(n, k) \in J}$, counted with their multiplicities and numbered in non-decreasing order: $\la_{n_1 k_1} \le \la_{n_2 k_2}$ if $(n_1, k_1) < (n_2, k_2)$. The following asymptotic relation holds:
\begin{equation} \label{asymptla}
    \rho_{nk} := \sqrt{\la_{nk}} = n + r_k + \varkappa_{nk}, \quad (n, k) \in J, \quad \{ \varkappa_{nk} \} \in l_2,
\end{equation}
where
\begin{equation} \label{defJ} 
J := \{ (n, k) \colon n \in \mathbb N, \, k = \overline{1, m} \} \cup \{ (0, k) \colon k = \overline{p^{\perp} + 1, m} \}, \quad p^{\perp} := \dim(\Ker T_1 \cap \Ker T_2),
\end{equation}
$\{ r_k \}_{k = 1}^m$ are the zeros of the function $\det(W^0(\rho))$ on $[0, 1)$, $0 \le r_1 \le r_2 \le \dots \le r_m < 1$,
\begin{equation} \label{defW0}
W^0(\rho) := (T_2 T_1 + T_2^{\perp} T_1^{\perp}) \sin \rho \pi + (T_2^{\perp} T_1 - T_2 T_1^{\perp}) \cos \rho \pi.
\end{equation}

\end{prop}

The {\it Weyl solution} of $L$ is the matrix solution $\Phi(x, \la)$ of equation~\eqref{eqv} satisfying the boundary conditions $V_1(\Phi) = I$, $V_2(\Phi) = 0$. The matrix function $M(\la) := T_1 \Phi(0, \la) + T_1^{\perp} \Phi^{[1]}(0, \la)$ is called the {\it Weyl matrix} of $L$. The matrix functions $M(\la)$ and $\Phi(x, \la)$ for each fixed $x \in [0, \pi]$ are meromorphic in $\la$. All their singularities are the simple poles at $\la = \la_{nk}$, $(n, k) \in J$. Denote
$$
\al_{nk} := \Res_{\la = \la_{nk}} M(\la), \quad (n, k) \in J.
$$

The matrices $\{ \al_{nk} \}_{(n, k) \in J}$ are called the \textit{weight matrices} and the collection $\{ \la_{nk}, \al_{nk} \}_{(n, k) \in J}$ is called the \textit{spectral data} of $L$.

Let $\la_{n_1 k_1} = \la_{n_2 k_2} = \dots = \la_{n_r k_r}$ be a group of multiple eigenvalues maximal by inclusion, $(n_1, k_1) < (n_2, k_2) < \dots < (n_r, k_r)$. Clearly, $\al_{n_1 k_1} = \al_{n_2 k_2} = \dots = \al_{n_r k_r}$. Define $\al'_{n_1 k_1} := \al_{n_1 k_1}$, $\al_{n_j k_j} := 0$, $j = \overline{2, r}$. We obtain the sequences of matrices $\{ \al'_{nk} \}_{(n, k) \in J}$.

\begin{prop} \label{prop:asymptal}
The weight matrices are Hermitian non-negative definite: $\al_{nk} = \al_{nk}^{\dagger} \ge 0$, $(n, k) \in J$. For each $(n, k) \in J$, $\rank(\al_{nk})$ equals the multiplicity of the eigenvalue $\la_{nk}$. Furthermore, the asymptotic relation holds:
\begin{equation} \label{asymptal}
\al_n^{(k)} := \sum_{r_s \in J_k} \al'_{ns} = \frac{2}{\pi} (T_1 + n T_1^{\perp}) (A_k + K_{nk}) (T_1 + n T_1^{\perp}), \quad n \ge 1, \quad k \in \mathcal J,
\end{equation}
where 
\begin{gather*}
\mathcal J := \{ 1 \} \cup \{ k = \overline{2, m} \colon r_k \ne r_{k-1} \}, \quad J_k := \{ s = \overline{1, m} \colon r_s = r_k \}, \quad \{ \| K_{nk} \| \} \in l_2, \\
A_k := \pi \Res_{\rho = r_k} (W^0(\rho))^{-1} U^0(\rho), \\
U^0(\rho) := (T_2 T_1 + T_2^{\perp} T_1^{\perp}) \cos \rho \pi + (T_2 T_1^{\perp} - T_2^{\perp} T_1) \sin \rho \pi,
\end{gather*}
The matrices $\{ A_k \}_{k \in \mathcal J}$ are orthogonal projection matrices having the following properties:
$$
\rank(A_k) = |J_k|, \qquad A_k A_s = 0, \: k \ne s, \qquad \sum_{k \in \mathcal J} A_k = I.
$$
\end{prop}

Consider a group of multiple eigenvalues $\la_{n_1 k_1} = \la_{n_2 k_2} = \dots = \la_{n_r k_r}$ maximal by inclusion. By Proposition~\ref{prop:asymptal}, we have $\rank (\al_{n_1 k_1}) = r$, so $\Ran \al_{n_1 k_1}$ is an $r$-dimensional subspace in $\mathbb C^m$. Choose a basis $\{ \chi_{n_j k_j} \}_{j = 1}^r$ in this subspace. This choice is non-unique. Proposition~\ref{prop:complete} is valid for any choice of the basis. Thus, we have defined the vector sequence $\{ \chi_{nk} \}_{(n, k) \in J}$. Consider the sequence of vector functions
\begin{equation} \label{defX}
\mathcal X := \{ X_{nk} \}_{(n, k) \in J}, \quad
X_{nk}(x) := \left( \cos (\rho_{nk}x) T_1 + \frac{\sin (\rho_{nk}x)}{\rho_{nk}} T_1^{\perp} \right) \chi_{nk}.
\end{equation}

\begin{prop} \label{prop:complete}
The sequence $\mathcal X$ is complete in $L_2((0, \pi); \mathbb C^m)$.
\end{prop}

Proposition~\ref{prop:complete} immediately follows from \cite[Theorem~5.1]{Bond-dir}, which asserts the completeness of the following sequence $\mathcal Y$.
Put
$$
T_{nk} := \left\{ \begin{array}{ll}
                T_1 + \rho_{nk} T_1^{\perp}, \quad & \rho_{nk} \ne 0, \\
                I, \quad & \rho_{nk} = 0,
          \end{array} \right.
\qquad          
B_{nk} := \frac{\pi}{2} T_{nk}^{-1} \al_{nk} T_{nk}^{-1}, \quad (n, k) \in J.
$$
Clearly, $\rank(B_{nk})$ equals the multiplicity of the eigenvalue $\la_{nk}$.

For any group of multiple eigenvalues $\la_{n_1 k_1} = \la_{n_2 k_2} = \dots = \la_{n_r k_r}$ considered above, choose an orthonormal basis $\{ \mathcal E_{n_j k_j} \}_{j = 1}^r$ in the $r$-dimensional subspace $\Ran B_{n_1 k_1}$. Thus, we have defined the vector sequence $\{ \mathcal E_{nk} \}_{(n, k) \in J}$. Define
\begin{equation} \label{defY}
\mathcal Y := \{ Y_{nk} \}_{(n, k) \in J}, \quad
Y_{nk}(x) := \left\{ \begin{array}{ll}
                (\cos (\rho_{nk}x) T_1 + \sin (\rho_{nk}x) T_1^{\perp}) \mathcal E_{nk}, \quad & \rho_{nk} \ne 0, \\
                (T_1 + x T_1^{\perp}) \mathcal E_{nk}, \quad & \rho_{nk} = 0.
            \end{array}\right.
\end{equation}
Clearly, the completeness of $\mathcal X$ is equivalent to the completeness of $\mathcal Y$ independently of the choice of the bases in the corresponding subspaces.

Now we turn to discuss the following inverse spectral problem.

\begin{ip} \label{ip:1}
Given the spectral data $\{ \la_{nk}, \al_{nk} \}_{(n, k) \in J}$, find $\sigma$, $T_1$, $T_2$, $H_2$.
\end{ip}

Along with the problem $L$, we consider the problem $\tilde L = L(\tilde \sigma, \tilde T_1, \tilde T_2, \tilde H_2)$ of the same form but with different coefficients. We agree that if a symbol $\ga$ denotes an object related to $L$, then the symbol $\tilde \ga$ with tilde denotes the similar object related to $\tilde L$. Note that the quasi-derivatives for these two problems are supposed to be different: $Y^{[1]} = Y' - \sigma Y$ for $L$ and $Y^{[1]} = Y' - \tilde \sigma Y$ for $\tilde L$. In \cite{Bond-dir}, the following uniqueness theorem has been obtained.

\begin{prop} \label{prop:uniq}
If $\la_{nk} = \tilde \la_{nk}$, $\al_{nk} = \tilde \al_{nk}$, $(n, k) \in J$, $J = \tilde J$, then 
\begin{equation} \label{transL}
\sigma(x) = \tilde \sigma(x) + H_1^{\diamond} \:\: \text{a.e. on} \:\: (0, \pi), \quad T_1 = \tilde T_1, \quad T_2 = \tilde T_2, \quad H_2 = \tilde H_2 - T_2 H_1^{\diamond} T_2,
\end{equation}
where 
\begin{equation} \label{H1perp}
    H_1^{\diamond} = (H_1^{\diamond})^{\dagger} = T_1^{\perp} H_1^{\diamond} T_1^{\perp}.
\end{equation}
Thus, the spectral data $\{ \la_{nk}, \al_{nk} \}_{(n, k) \in J}$ uniquely specify the problem $L$ up to a transform~\eqref{transL} given by an arbitrary matrix $H_1^{\diamond}$ satisfying~\eqref{H1perp}. Conversely, the transform~\eqref{transL} does not change the spectral data.
\end{prop}

The main result of this paper is the following theorem, which provides the characterization of the spectral data $\{ \la_{nk}, \al_{nk} \}_{(n, k) \in J}$ of the problem $L$.

\begin{thm} \label{thm:nsc}
Let $T_1, T_2 \in \mathbb C^{m \times m}$ be arbitrary fixed orthogonal projection matrices. Then, for a collection $\{ \la_{nk}, \al_{nk} \}_{(n, k) \in J}$ to be the spectral data of a problem $L = L(\sigma, T_1, T_2, H_2)$ in the form~\eqref{eqv}-\eqref{bc}, the following conditions are necessary and sufficient:

(i) $\la_{nk} \in \mathbb R$, $\al_{nk} \in \mathbb C^{m \times m}$, $\al_{nk} = \al_{nk}^{\dagger} \ge 0$, $\rank (\al_{nk})$ is equal to the multiplicity of the corresponding value $\la_{nk}$ (i.e., to the number of times $\la_{nk}$ occurs in the sequence), for all $(n, k) \in J$, and $\al_{nk} = \al_{ls}$ if $\la_{nk} = \la_{ls}$.

(ii) The asymptotic relations~\eqref{asymptla} and~\eqref{asymptal} hold, where $\{ r_k \}_{k = 1}^m$ and $\{ A_k \}_{k \in \mathcal J}$ are defined as in Propositions~\ref{prop:asymptla} and~\ref{prop:asymptal}, respectively, by using the fixed $T_1$ and $T_2$.

(iii) The sequence $\mathcal X$ defined by~\eqref{defX} is complete in $L_2((0, \pi); \mathbb C^m)$.
\end{thm}

In Theorem~\ref{thm:nsc}, the index set $J$ is defined by the fixed matrices $T_1$ and $T_2$ via~\eqref{defJ}.
We suppose that the matrices $T_1$ and $T_2$ are initially given, but this is done only for convenience of formulation.
In fact, $T_1$ and $T_2$ can be uniquely recovered from the spectral data $\{ \la_{nk}, \al_{nk} \}_{(n, k) \in J}$ by \cite[Algorithm~6.7]{Bond-dir}. Note that, in condition (iii), the sequence $\mathcal X$ depends on the choice of $\{ \chi_{nk} \}$. Obviously, if condition (iii) holds for some choice of $\{ \chi_{nk} \}$, then it holds for any possible choice of $\{ \chi_{nk} \}$.

The necessity part of Theorem~\ref{thm:nsc} readily follows from Propositions~\ref{prop:asymptla}-\ref{prop:complete}. Therefore, our goal is to prove the sufficiency part. For this purpose, we need one more proposition proved in \cite{Bond-dir}.

\begin{prop} \label{prop:Riesz}
Suppose that the data $\{ \la_{nk}, \al_{nk} \}_{(n, k) \in J}$ satisfy the conditions (i)-(iii) of Theorem~\ref{thm:nsc}. Then the sequence $\mathcal Y$ constructed by~\eqref{defY} is a Riesz basis in $L_2((0, \pi); \mathbb C^{m \times m})$.
\end{prop}

The proof of Theorem~\ref{thm:nsc} is based on several auxiliary theorems and lemmas provided in Sections~3-6. This proof is constructive and yields Algorithm~\ref{alg:1} for solving Inverse Problem~\ref{ip:1}.

\section{Main equation}

The goal of this section is to reduce the nonlinear Inverse Problem~\ref{ip:1} to the linear \textit{main equation} in a special Banach space. For construction of this Banach space, we group the eigenvalues with respect to their asymptotics~\eqref{asymptla}. In the next sections, the main equation is used for the proof of Theorem~\ref{thm:nsc} and for constructive solution of the inverse problem. 

Consider the problem $L = L(\sigma, T_1, T_2, H_2)$ with the spectral data $\{ \la_{nk}, \al_{nk} \}_{(n, k) \in J}$. Without loss of generality, we may assume that $\la_{nk} \ge 0$ and $\rho_{nk} = \sqrt{\la_{nk}} \ge 0$, $(n, k) \in J$. One can easily achieve this condition by a shift:
$$
\sigma(x) := \sigma(x) + c x I, \quad H_2 := H_2 - c \pi T_2, \quad \la_{nk} := \la_{nk} + c, \quad c > 0.
$$

Fix the model problem $\tilde L := L(0, T_1, T_2, 0)$. We have
\begin{gather} \label{tvv}
\tilde \vv(x, \la) = \cos \rho x \, T_1 + \frac{\sin \rho x}{\rho} T_1^{\perp}, \quad  
\tilde \rho_{nk} = n + r_k, \\ \label{tal}
\tilde \al_{nk} = \begin{cases}
                    \frac{2}{\pi} (T_1 + \tilde \rho_{nk} T_1) A_k (T_1 + \tilde \rho_{nk} T_1), \quad \tilde \rho_{nk} \ne 0, \\
                    \frac{1}{\pi} T_1 A_k T_1, \quad \tilde \rho_{nk} = 0.
                  \end{cases}
\end{gather}

Denote $\langle Z, Y \rangle := Z Y^{[1]} - Z^{[1]} Y$.
Introduce the notations
\begin{gather} \label{defD}
\tilde D(x, \la, \mu) := \frac{\langle \tilde \vv(x, \la), \tilde \vv(x, \mu)\rangle}{\la - \mu} = \int_0^x \tilde \vv(t, \la) \tilde \vv(x, \mu) \, d\mu, \\ \nonumber
\la_{nk0} := \la_{nk}, \quad \la_{nk1} := \tilde \la_{nk}, \quad \rho_{nk0} := \rho_{nk}, \quad \rho_{nk1} := \tilde \rho_{nk}, \\ \nonumber
\al_{nk0} := \al_{nk}, \quad \al_{nk1} := \tilde \al_{nk}, \quad
\al'_{nk0} := \al'_{nk}, \quad \al'_{nk1} := \tilde \al'_{nk}.
\end{gather}

Using the contour integration in the $\la$-plane (see \cite{Bond-scal}), we prove the following lemma.

\begin{lem} \label{lem:cont}
The following relation holds
\begin{equation} \label{cont}
    \tilde \vv(x, \la_{nki}) = \vv(x, \la_{nki}) + \sum_{(l, s) \in J} (\vv(x, \la_{ls0}) \al'_{ls0} \tilde D(x, \la_{ls0}, \la_{nki}) - \vv(x, \la_{ls1}) \al'_{ls1} \tilde D(x, \la_{ls1}, \la_{nki})),
\end{equation}
for $(n, k) \in J$, $i = 0, 1$. 
The series converges in the sense $\lim\limits_{N \to \iy} \sum\limits_{l \le N} (\dots)$ absolutely and uniformly by $x \in [0, \pi]$.
\end{lem}

It is inconvenient to use~\eqref{cont} as the main equation of the inverse problem, since the series in \eqref{cont} only converges ``with brackets''. Below we transform~\eqref{cont} into a linear equation in a specially constructed Banach space.

Let $\mathcal J =: \{ j_k \}_{k = 1}^{|\mathcal J|}$. Divide the square roots $\{ \rho_{nki} \}$ of the eigenvalues into collections (multisets) as follows:
\begin{equation} \label{defG}
\begin{array}{rl}
G_1 & := \{ \rho_{nki} \colon (n, k) \in J, n \le n_0, i = 0, 1 \},
\\ G_{|\mathcal J|q + s + 1} & := \{ \rho_{nki} \colon n = n_0 + q + 1, r_k = r_{j_s}, i = 0, 1 \}, \quad q \ge 0, \quad s = \overline{1, |\mathcal J|}.
\end{array}
\end{equation}
In view of asymptotics~\eqref{asymptla}, we can choose and fix $n_0$ such that $G_n \cap G_k = \varnothing$ for $n \ne k$.

For any multiset $\mathcal G$ of real numbers, let $B(\mathcal G)$ be the finite-dimensional space of matrix functions $f \colon \mathcal G \to \mathbb C^{m \times m}$ such that $f(\rho) = f(\theta)$ if $\rho = \theta$. The norm in $B(\mathcal G)$ is defined as follows:
\begin{equation} \label{normBG}
   \| f \|_{B(\mathcal G)} = \max \left\{ \max_{\rho \in \mathcal G} \| f(\rho) \|, \max_{\substack{\rho \ne \theta \\ \rho, \theta \in \mathcal G}} 
   |\rho - \theta|^{-1} \| f(\rho) - f(\theta) \| \right\}.
\end{equation}

Introduce the Banach space $\mathfrak B$ of infinite sequences:
\begin{equation} \label{defB}
\mathfrak B := \{ f = \{ f_n \}_{n \ge 1} \colon f_n \in B(G_n), n \ge 1, \| f \|_{\mathfrak B} := \sup_{n \ge 1} \| f_n \|_{B(G_n)} < \iy \}.
\end{equation}

For $(n, k) \in J$, $i = 0, 1$, denote
$$
T_{nki} := \left\{ \begin{array}{ll}
                T_1 + T_1^{\perp} \rho_{nki}, \quad & \rho_{nki} \ne 0, \\
                I, \quad & \rho_{nki} = 0,
            \end{array} \right.
\qquad \phi_{nki}(x) := \vv(x, \la_{nki}) T_{nki}.
$$
Put $\phi(x) := \{ \phi_n(x) \}_{n = 1}^{\iy}$, $\phi_n(x)(\rho_{lsj}) := \phi_{lsj}(x)$ for $\rho_{lsj} \in G_n$. Analogously, define $\tilde \phi(x)$ replacing $\vv$ by $\tilde \vv$. Using relation~\eqref{asymptvv}, we obtain the estimates
$$
\| \phi_{nki} (x) \| \le C, \quad \| \phi_{nki}(x) - \phi_{lsj}(x) \| \le C |\rho_{nki} - \rho_{lsj}|, \quad \rho_{nki}, \rho_{lsj} \in G_q, 
$$
for $q \ge 1$, $x \in [0, \pi]$. Hence, $\phi(x) \in \mathfrak B$ and, similarly, $\tilde \phi(x) \in \mathfrak B$ for each fixed $x \in [0, \pi]$. In addition, $\phi(x)$ and $\tilde \phi(x)$ are uniformly bounded in $\mathfrak B$ with respect to $x \in [0, \pi]$.

For each fixed $x \in [0, \pi]$, define the linear operator $\tilde R(x) \colon \mathfrak B \to \mathfrak B$, $\tilde R(x) = [\tilde R_{k, n}(x)]_{k, n = 1}^{\iy}$, acting on an element $f = \{ f_n \}_{n = 1}^{\iy}$ of $\mathfrak B$ by the following rule:
\begin{gather} \label{defR1}
    (f \tilde R(x))_n = \sum_{k = 1}^{\iy} f_k \tilde R_{k, n}(x), \quad \tilde R_{k, n}(x) \colon B(G_k) \to B(G_n), \\ \label{defR2}
    (f_k \tilde R_{k, n}(x))(\rho_{\eta q i}) = \sum_{\rho_{lsj} \in G_k} (-1)^j f_k(\rho_{lsj}) T_{lsj}^{-1} \al'_{lsj} \tilde D(x, \la_{lsj}, \la_{\eta q i}) T_{\eta q i} 
    \quad \rho_{\eta q i} \in G_n. 
\end{gather}
Here we put operators to the right of operands to show the order of matrix multiplication.

\begin{thm} \label{thm:RB}
The series \eqref{defR1} converges in the $B(G_n)$-norm.
For each fixed $x \in [0, \pi]$, the operator $\tilde R(x)$ is bounded and can be approximated by finite-dimensional operators in the norm $\| . \|_{\mathfrak B \to \mathfrak B}$.
\end{thm}

Theorem~\ref{thm:RB} is proved in Section~4. Taking the above definitions into account, we rewrite relation~\eqref{cont} in the form
\begin{equation} \label{main}
\phi(x) (\mathcal I + \tilde R(x)) = \tilde \phi(x), \quad x \in [0, \pi].
\end{equation}
where $\mathcal I$ is the unit operator in $\mathfrak B$.
For each fixed $x \in [0, \pi]$, relation \eqref{main} is a linear equation with respect to $\phi(x)$ in the Banach space $\mathfrak B$. 
Note that $\tilde \phi(x)$ and $\tilde R(x)$ are constructed by the spectral data $\{ \la_{nk}, \al_{nk} \}_{(n, k) \in J}$ and by the model problem $\tilde L$, while the unknown element $\phi(x)$ is related to the problem $L$. Further, the solution of the main equation~\eqref{main} is used for constructive solution of Inverse Problem~\ref{ip:1}.
Therefore, we call \eqref{main} the \textit{main equation} of the inverse problem.

\section{Estimates}

In this section, we investigate properties of the operator $\tilde R(x)$ and obtain the estimates needed in further proofs. It is supposed that $\tilde R(x)$ is the operator constructed by the collection $\{ \la_{nk}, \al_{nk} \}_{(n, k) \in J}$ satisfying the asymptotics~\eqref{asymptla} and~\eqref{asymptal} and by the problem $\tilde L = L(0, T_1, T_2, 0)$. We emphasize that $\{ \la_{nk}, \al_{nk} \}_{(n, k) \in J}$ are not assumed to be the spectral data of some problem $L$. This allows us to use the results of this section in the proof of the sufficiency in Theorem~\ref{thm:nsc}.

For $n \ge 1$, denote
\begin{gather} \label{defalG}
\hat \al(G_n) := \sum_{\rho_{lsj} \in G_n} (-1)^j T_{lsj}^{-1} \al'_{lsj} T_{lsj}^{-1},  \\ \label{defxi}
\xi_n := \sum_{\rho, \theta \in G_n} |\rho - \theta| + \| \hat \al(G_n) \|.
\end{gather}
It follows from the asymptotic formulas~\eqref{asymptla} and~\eqref{asymptal} that
\begin{equation} \label{Xi}
\Xi := \left( \sum_{n = 1}^{\iy} \xi_n^2 \right)^{1/2} < \iy.
\end{equation}

Put 
$$
\tilde D_T (x, \rho, \theta) := (T_1 + \rho T_1^{\perp})\tilde D(x, \rho^2, \theta^2) (T_1 + \theta T_1^{\perp}).
$$
Substituting~\eqref{tvv} into~\eqref{defD} and using~\eqref{asymptla}, \eqref{defxi}, we obtain the following lemma.

\begin{lem} \label{lem:estD}
For $x \in [0, \pi]$, $n, k \ge 1$, $\rho, \zeta \in G_n$, $\theta, \chi \in G_k$, the following estimates hold
\begin{gather*}
    \| \tilde D_T(x, \rho, \theta) \| \le \frac{C}{|n - k| + 1}, \quad
    \| \tilde D_T(x, \rho, \theta) - \tilde D_T(x, \rho, \chi) \| \le \frac{C \xi_k}{|n - k| + 1}, \\
    \| \tilde D_T(x, \rho, \theta) - \tilde D_T(x, \rho, \chi) - \tilde D_T(x, \zeta, \theta) + \tilde D_T(x, \zeta, \chi) \| \le \frac{C \xi_n \xi_k}{|n - k| + 1}.
\end{gather*}
For $x \in [0, \pi]$, $n \ge 1$, $\rho, \zeta \in G_n$, $\theta \in \mathbb C$, we have
$$
\| \tilde D_T(x, \rho, \theta) \| \le \frac{C \exp(|\mbox{Im}\, \theta|x)}{|\theta - m_n| + 1}, \quad 
\| \tilde D_T(x, \rho, \theta) - \tilde D_T(x, \zeta, \theta) \| \le \frac{C \xi_n \exp(|\mbox{Im}\, \theta|x)}{|\theta - m_n| + 1},
$$
where $m_n = l + r_s$, $(l, s) \colon \rho_{ls0} \in G_n$.
In all the estimates, the constant $C$ does not depend on $n$, $k$, $x$, etc.
\end{lem}

\begin{lem}
For $x \in [0, \pi]$, the following estimates hold
\begin{gather} \label{estRkn}
\| \tilde R_{k, n}(x) \|_{B(G_k) \to B(G_n)} \le \frac{C \xi_k}{|n - k| + 1}, \quad n, k \ge 1, \\ \label{estR}
\| \tilde R(x) \|_{\mathfrak B \to \mathfrak B} \le C \Xi,
\end{gather}
where the constant $C$ does not depend on $n$, $k$ and $x$.
\end{lem}

\begin{proof}
Estimate~\eqref{estRkn} is proved by using~\eqref{normBG}, \eqref{defR2}, and the summation rule
\begin{equation} \label{group}
\sum_{u = 1}^v a_u b_u c_u  = \sum_{u = 1}^v (a_u - a_1) b_u c_u + a_1 \sum_{u = 1}^v b_u (c_u - c_1) + a_1 \sum_{u = 1}^v b_u \, c_1.
\end{equation}
We put
$$
a_u = f_k(\rho_{lsj}), \quad b_u = (-1)^j T_{lsj}^{-1} \al'_{lsj} T_{lsj}^{-1}, \quad c_u = T_{lsj} \tilde D(x, \la_{lsj}, \la_{\eta q i}) T_{\eta q i},
$$
apply the estimates $\| f_k \|_{B(G_k)} \le C$, \eqref{defxi} and Lemma~\ref{lem:estD}, and so arrive at~\eqref{estRkn}.

Relations~\eqref{defB} and \eqref{defR1} yield
$$
\| \tilde R(x) \|_{\mathfrak B \to \mathfrak B} \le \sup_{n \ge 1} \sum_{k = 1}^{\iy} \| \tilde R_{k,n}(x) \|_{B(G_k) \to B(G_n)}.
$$
Using~\eqref{Xi} and~\eqref{estRkn}, we arrive at~\eqref{estR}.
\end{proof}

\begin{proof}[Proof of Theorem~\ref{thm:RB}]
It readily follows from~\eqref{Xi}, \eqref{estRkn}, and~\eqref{estR}
that the series \eqref{defR1} converges in $B(G_n)$-norm and the operator $\tilde R(x)$ is bounded. Define the finite-dimensional operators
$$
\tilde R^N(x) = [\tilde R_{k, n}^N(x)]_{k, n = 1}^{\iy}, \quad \tilde R_{k, n}^N = \left\{
\begin{array}{ll}
\tilde R_{k, n}(x), \quad & k \le N, \\
0, \quad & k > N,
\end{array}\right.
\quad N \ge 1.
$$
Using~\eqref{estRkn}, it is easy to show that the sequence $\{ \tilde R^N(x) \}$ converges to $\tilde R(x)$ in the norm $\| .\|_{\mathfrak B \to \mathfrak B}$ for each fixed $x \in [0, \pi]$.
\end{proof}

Further we need the following auxiliary proposition, which easily follows from asymptotics~\eqref{asymptla} and the Riesz-basicity of the sequences $\{ \cos (n + \varkappa_n) x \}_{n = 0}^{\iy}$, $\{ \sin (n + \varkappa_n) x \}_{n = 1}^{\iy}$ in $L_2(0, \pi)$, $\{ \varkappa_n \} \in l_2$, $n + \varkappa_n \ne k + \varkappa_k$ for $n \ne k$ (see \cite{He01}).

\begin{prop} \label{prop:cos}
(i) Let $\{ \varkappa_{nki} \}$ be an arbitrary sequence from $l_2$. Then the series 
$$
F_c(x) := \sum_{n,k,i} \varkappa_{nki} \cos (\rho_{nki} x), \quad F_s(x) := \sum_{n,k,i} \varkappa_{nki} \sin (\rho_{nki} x)
$$
converge in $L_2(0, \pi)$ and 
$$
\| F_c \|_{L_2(0, \pi)}, \| F_s \|_{L_2(0, \pi)} \le C \| \{ \varkappa_{nki}\}\|_{l_2},
$$
where the constant $C$ depends only on $\{ \rho_{nki} \}$ and not on $\{ \varkappa_{nki} \}$. 

(ii) Let $F(x)$ be arbitrary function from $L_2(0, \pi)$. Put
$$
\varkappa_{c, nki} = \int_0^{\pi} F(x) \cos (\rho_{nki} x) \, dx, \quad
\varkappa_{s, nki} = \int_0^{\pi} F(x) \sin (\rho_{nki} x) \, dx.
$$
Then the sequences $\{ \varkappa_{c, nki} \}$ and $\{ \varkappa_{s, nki} \}$ belong to $l_2$ and
$$
    \| \{ \varkappa_{c, nki} \} \|_{l_2},  \| \{ \varkappa_{s, nki} \} \|_{l_2} \le C \| F \|_{L_2(0, \pi)},
$$
where the constant $C$ depends only on $\{ \rho_{nki} \}$ and not on $F$.

In (i) and (ii), the indices $(n, k, i)$ run over the set: $(n, k) \in J$, $i = 0, 1$.
\end{prop}

For convenience, for any sequence $f = \{ f_n \}_{n = 1}^{\iy} \in \mathfrak B$, we denote $f_{lsj} := f_n(\rho_{lsj})$, where $n$ is such that $\rho_{lsj} \in G_n$.

\begin{lem} \label{lem:Rl2}
Suppose that $f \in \mathfrak B$ and $g(x) = \tilde R(x) f$. Then the corresponding sequence $\{ \| g_{nki}(x) \| \}$ belongs to $l_2$ for each fixed $x \in [0, \pi]$ and 
$$
\| \{ \| g_{nki}(x) \|\}\|_{l_2} \le C \| f \|_{\mathfrak B},
$$
uniformly by $x \in [0, \pi]$.
\end{lem}

\begin{proof}
Substituting~\eqref{tvv} into \eqref{defD}, we obtain
\begin{equation} \label{sm1}
\tilde D(x, \rho_{lsj}, \rho_{nki}) = T_{lsj}^{-1} \int_0^x (\cos (\rho_{lsj} t) \cos (\rho_{nki} t) \, T_1 + \sin (\rho_{lsj} t) \sin (\rho_{nki} t)  \, T_1^{\perp}) \, dx \, T_{nki}^{-1}. 
\end{equation}
For simplicity, throughout this proof we assume that $\rho_{nki} \ne 0$ and $\rho_{lsj} \ne 0$. The opposite case requires minor technical changes. Using~\eqref{defR2} and \eqref{sm1}, we derive
\begin{gather} \label{sm2}
g_{nki}(x) T_1 = \int_0^x F(t) \cos (\rho_{nki} t) \, dt, \\ \label{defF}
F(t) := \sum_{l, s, j} (-1)^j f_{lsj} T_{lsj}^{-1} \al'_{lsj} T_{lsj}^{-1} T_1 \cos (\rho_{lsj} t).
\end{gather}
Using the summation rule~\eqref{group} in \eqref{defF}, relations~\eqref{asymptla}, \eqref{defxi}, and Proposition~\ref{prop:cos}(i), we prove that $F \in L_2((0, \pi); \mathbb C^{m \times m})$ and $\| F \|_{L_2} \le C \| f \|_{\mathfrak B}$.
Applying Proposition~\ref{prop:cos}(ii) to~\eqref{sm2}, we show that $\{ \| g_{nki}(x) T_1 \| \} \in l_2$ for each fixed $x \in [0, \pi]$ and the $l_2$-norm of this sequence does not exceed $C \| F \|_{L_2}$, where $C$ does not depend on $x \in [0, \pi]$. Similar arguments are valid for $g_{nki} T_1^{\perp}$. This concludes to proof.
\end{proof}

\begin{lem} \label{lem:Rl1}
Suppose that $f \in \mathfrak B$ and $g(x) = \tilde R(x) f$. Let indices $q \in \mathcal J$, $k, s \in J_q$, $i, j \in \{ 0, 1 \}$ be fixed. Then the sequence $\{ \| g_{nki} - g_{nsj} \| \}$ belongs to $l_1$ for each fixed $x \in [0, \pi]$ and
$$
\| \{ \| g_{nki}(x) - g_{nsj}(x) \| \} \|_{l_1} \le C \| f \|_{\mathfrak B}
$$
uniformly by $x \in [0, \pi]$.
\end{lem}

\begin{proof}
For fixed $k$, $s$, $i$, $j$ satisfying the conditions of the lemma, we have
\begin{equation} \label{difcos}
\cos (\rho_{nki} t) - \cos (\rho_{nsj} t) = \varkappa_n t \sin (n + r_k) t + \zeta_n(t), \quad \{ \varkappa_n \} \in l_2, \quad \sum_n \max_{t \in [0, \pi]}|\zeta_n(t)| \le C,
\end{equation}
where $\varkappa_n$ does not depend on $t$, $t \in [0, \pi]$. Using~\eqref{sm2}, \eqref{difcos}, and Proposition~\ref{prop:cos}, we obtain
$$
    (g_{nki}(x) - g_{nsj}(x)) T_1 = \int_0^x F(t) (\cos (\rho_{nki} t) - \cos (\rho_{nsj} t)) \, dt = \varkappa_n K_n(x) + Z_n(x),
$$
where
$$
    \| \{ \| K_n(x) \| \} \|_{l_2} \le C \| F \|_{L_2}, \quad \| Z_n(x) \| \le C \| F \|_{L_2} \max_{t \in [0, \pi]} |\zeta_n(t)|
$$
uniformly by $x \in [0, \pi]$. Taking the estimate $\| F \|_{L_2} \le C \| f \|_{\mathfrak B}$ into account, we obtain the assertion of the lemma for the sequence $\{ \| (g_{nki}(x) - g_{nsj}(x)) T_1 \| \}$. The sequence $\{ \| (g_{nki}(x) - g_{nsj}(x)) T_1^{\perp} \| \}$ can be studied similarly.
\end{proof}

\section{Solvability of main equation}

In this section, we suppose that the collection $\{ \la_{nk}, \al_{nk} \}_{(n, k) \in J}$ satisfies the conditions of Theorem~\ref{thm:nsc} and prove the unique solvability of the main equation~\eqref{main}.

\begin{thm} \label{thm:solve}
For each fixed $x \in [0, \pi]$, the operator $(\mathcal I + \tilde R(x)) \colon \mathfrak B \to \mathfrak B$ has a bounded inverse, so the main equation~\eqref{main} has a unique solution $\phi(x) \in \mathfrak B$.
\end{thm}

\begin{proof}
Fix $x \in [0, \pi]$. By virtue of Theorem~\ref{thm:RB}, the operator $\tilde R(x)$ can be approximated by finite-dimensional operators. Therefore, in view of Fredholm's Theorem, it suffices to prove that the homogeneous equation
\begin{equation} \label{homo}
 \beta(x) (\mathcal I + \tilde R(x)) = 0, \quad \beta(x) = \{ \beta_n(x) \}_{n = 1}^{\iy} \in \mathfrak B,
\end{equation}
has the only solution $\beta(x) = 0$ in $\mathfrak B$. Since $\beta(x) = -\tilde R(x) \beta(x)$, Lemmas~\ref{lem:Rl2} and~\ref{lem:Rl1} imply
\begin{equation} \label{betaseq}
\{ \| \beta_{nki}(x) \| \} \in l_2, \quad \{ \| \beta_{nki}(x) - \beta_{nsj}(x) \|\} \in l_1,
\end{equation}
for fixed $k, s \in J_q$, $q \in \mathcal J$, $i, j \in \{ 0, 1 \}$. Introduce the matrix functions:
\begin{align} \label{defga}
    \ga(x, \la) & := -\sum_{l, s, j} (-1)^j \beta_{lsj}(x) T_{lsj}^{-1} \al'_{lsj} \tilde D(x, \la_{lsj}, \la), \\ \label{defGa}
    \Gamma(x, \la) & := -\sum_{l, s, j} (-1)^j \beta_{lsj}(x) T_{lsj}^{-1} \al'_{lsj} \tilde E(x, \la_{lsj}, \la), \\ \nonumber
    \tilde E(x, \la, \mu) & := \frac{\langle \tilde \vv(x, \la), \tilde \Phi(x, \mu)\rangle}{\la - \mu}, \\ \label{funcB}
    \mathscr B(x, \la) & := \Gamma(x, \la) (\ga(x, \overline{\la}))^{\dagger}.
\end{align}
In \eqref{defga} and~\eqref{defGa}, the indices $(l, s, j)$ run over the set: $(l, s) \in J$, $j = 0, 1$. The matrix function $\ga(x, \la)$ is entire in $\la$, while $\Gamma(x, \la)$ and $\mathscr B(x, \la)$ are meromorphic in $\la$ with the simple poles $\{ \la_{nki} \}$. Relation~\eqref{homo} implies $\ga(x, \la_{nki}) = \beta_{nki}(x) T_{nki}^{-1}$, $(n, k) \in J$, $i = 0, 1$. Calculations show that
\begin{equation} \label{ResB}
\Res_{\la = \la_{nk0}} \mathscr B(x, \la) = \ga(x, \la_{nk0}) \al_{nk0} (\ga(x, \la_{nk0}))^{\dagger}, \quad \Res_{\la = \la_{nk1}} \mathscr B(x, \la) = 0
\end{equation}
if $\la_{nk0} \ne \la_{ls1}$, $(n, k), (l, s) \in J$. The opposite case requires minor changes.

Using~\eqref{defga}, the summation rule~\eqref{group}, \eqref{betaseq}, \eqref{defxi}, and Lemma~\ref{lem:estD}, we obtain
\begin{equation} \label{estga}
\| \ga(x, \la) (T_1 + \rho T_1^{\perp}) \| \le C(x) \exp(|\mbox{Im}\,\rho|x) \sum_{n = 1}^{\iy} \frac{\theta_n}{|\rho - m_n| + 1}.
\end{equation}
Here and below, $\rho = \sqrt{\la}$, $\arg \rho \in [-\tfrac{\pi}{2}, \tfrac{\pi}{2})$, the notation $\{ \theta_n \}$ stands for various $l_1$-sequences of non-negative numbers. Analogously to~\eqref{estga}, we get
\begin{equation} \label{estGa}
    \| \Gamma(x, \la) (\rho T_1 + T_1^{\perp}) \| \le C(x) \exp(-|\mbox{Im}\,\rho|x) \sum_{n = 1}^{\iy} \frac{\theta_n}{|\rho - m_n| + 1}, \quad \rho \in G_{\de}, \quad |\rho| \ge \rho^*,
\end{equation}
where
$$
G_{\de} := \{ \rho \in \mathbb C \colon |\rho - (n + r_k)| \ge \de, n \in \mathbb Z, k = \overline{1, m} \},
$$
$\de$ and $\rho^*$ are some positive reals. Suppose that $\la \in \Upsilon_{N+r}$, $\Upsilon_{N+r} := \{ \la \in \mathbb C \colon |\la| = (N + r)^2 \}$, where $N \in \mathbb N$, $r$ is fixed, $r \ne r_k$, $k = \overline{1, m}$. Using~\eqref{funcB}, \eqref{estga}, and~\eqref{estGa}, we obtain
$$
\| \mathscr B(x, \la) \| \le \frac{C(x)}{N} \left( \sum_{n = 1}^{\iy} \frac{\theta_n}{|N - n| + 1}\right)^2.
$$
Consequently,
$$
\| \mathscr B(x, \la) \| \le \frac{C(x)}{N} f_N, \quad f_N := \sum_{n = 1}^{\iy} \frac{\theta_n}{(N - n + 1/2)^2}, \quad \la \in \Upsilon_{N + r}.
$$
Obviously, $\{ f_N \} \in l_1$. This implies
$$
\varliminf_{N \to \iy} \frac{f_N}{1/N} = 0.
$$
Hence, there exists a sequence $\{ N_k \}$ such that 
$$
\max_{\la \in \Upsilon N_k}\mathscr B(x, \la) = o(N_k^{-2}), \quad k \to \iy.
$$
Therefore,
$$
\lim_{k \to \iy} \oint_{\Upsilon_{N_k + r}} \mathscr B(x, \la) \, d\la = 0.
$$
Using the Residue Theorem and~\eqref{ResB}, we show that
$$
\sum_{(n, k) \in J} \ga(x, \la_{nk}) \al'_{nk} (\ga(x, \la_{nk}))^{\dagger} = 0.
$$
Since $\al'_{nk} = (\al'_{nk})^{\dagger} \ge 0$, we get 
\begin{equation} \label{sm3}
\ga(x, \la_{nk}) \al_{nk} = 0, \quad (n, k) \in J.
\end{equation}

It is easy to see that the matrix function $\ga(x, \rho^2) T_1$ is even and $\ga(x, \rho^2) \rho T_1^{\perp}$ is odd. It follows from~\eqref{estga}, \eqref{defga}, \eqref{betaseq}, \eqref{defxi}, \eqref{sm1}, and Proposition~\ref{prop:cos} that these matrix functions are $O(\exp(|\mbox{Im}\,\rho|x))$ and belong to $L_2(\mathbb R; \mathbb C^{m \times m})$. Applying the Paley-Wiener Theorem, we obtain the representation
\begin{equation} \label{sm4}
\ga(x, \la) = \int_0^{\pi} (h(x, t))^{\dagger} \left( \cos \rho t \, T_1 + \frac{\sin \rho t}{\rho} \, T_1^{\perp} \right) \, dt  \quad h(x, .) \in L_2((0, \pi); \mathbb C^m). 
\end{equation}
Combining~\eqref{sm3} and~\eqref{sm4}, we get $(h, X_{nk}) = 0$ for each fixed $x \in [0, \pi]$ and for all $(n, k) \in J$. Since the sequence $\mathcal X = \{ X_{nk} \}_{(n, k) \in J}$ is complete in $L_2((0, \pi); \mathbb C^m)$, it follows that $h = 0$ in $L_2((0, \pi); \mathbb C^m)$ for each fixed $x \in [0, \pi]$. Consequently, $\ga(x, \la) \equiv 0$ and $\beta(x) = 0$, so the homogeneous equation~\eqref{homo} has the unique solution in $\mathfrak B$. This yields the claim.
\end{proof}

\section{Proof of sufficiency}

In this section, we prove the sufficiency part of Theorem~\ref{thm:nsc}. Let $\{ \la_{nk}, \al_{nk} \}_{(n, k) \in J}$ be a collection satisfying the conditions of Theorem~\ref{thm:nsc}. Suppose that the Banach space $\mathfrak B$, the element $\tilde \phi(x) \in B$, and the operator $\tilde R(x) \colon \mathfrak B \to \mathfrak B$ for each fixed $x \in [0, \pi]$ are constructed in accordance with Section~3. By virtue of Theorem~\ref{thm:solve}, the main equation~\eqref{main} has the unique solution $\phi(x) \in B$ for each fixed $x \in [0, \pi]$. Similarly to \cite[Lemma~5.3]{Bond-scal}, we obtain the following result.

\begin{lem} \label{lem:psi}
The elements $\{ \phi_{nki}(x) \}$ of $\phi(x)$ can be represented in the form
$$
\phi_{nki}(x) = \cos (n + r_k) x \, T_1 + \sin (n + r_k) x \, T_1^{\perp} + \psi_{nki}(x), \quad (n, k) \in J, \quad i = 0, 1,
$$
where the matrix functions $\psi_{nki}$ are continuous on $[0, \pi]$, the sequence $\{ \| \psi_{nki}(x) \| \}$ belongs to $l_2$ for each fixed $x \in [0, \pi]$,  and the $l_2$-norm of this sequence is uniformly bounded by $x \in [0, \pi]$.
\end{lem}

Construct the matrix function $\sigma(x)$ and the matrix $H_2$ as follows:
\begin{align} \label{defsi}
    \sigma(x) := & -2 \sum_{n = 1}^{\iy} \Biggl( \sum_{\rho_{lsj} \in G_n} (-1)^j \phi_{lsj}(x) T_{lsj}^{-1} \al'_{lsj} T_{lsj}^{-1} \tilde \phi_{lsj}(x) - \frac{1}{2} \bigl( T_1 \hat \al(G_n) T_1 + T_1^{\perp} \hat \al(G_n) T_1^{\perp}\bigr)
    \Biggr)\\ \label{defH2}
    H_2 := & T_2 \sum_{n = 1}^{\iy} \Biggl( \sum_{\rho_{lsj} \in G_n} (-1)^j \phi_{lsj}(\pi) T_{lsj}^{-1} \al'_{lsj} T_{lsj}^{-1} \tilde \phi_{lsj}(\pi) - \bigl( T_1 \hat \al(G_n) T_1 + T_1^{\perp} \hat \al(G_n) T_1^{\perp}\bigr)
    \Biggr) T_2,
\end{align}
where $\hat \al(G_n)$ is defined by~\eqref{defalG}.

Relying on Lemmas~\ref{lem:Rl1} and~\ref{lem:psi}, Proposition~\ref{prop:cos}, and relations~\eqref{defxi}, \eqref{Xi}, we prove the following lemma.

\begin{lem} \label{lem:L2}
The series~\eqref{defsi} and~\eqref{defH2} converge in $L_2((0, \pi); \mathbb C^{m \times m})$ and $\mathbb C^{m \times m}$, respectively.    
\end{lem}

Thus, we have constructed $\sigma(x)$ and $H_2$ by formulas~\eqref{defsi} and~\eqref{defH2}, respectively. Consider the corresponding boundary value problem $L = L(\sigma, T_1, T_2, H_2)$ of the form~\eqref{eqv}-\eqref{bc}. It remains to prove the following theorem.

\begin{thm} \label{thm:sd}
The values $\{ \la_{nk}, \al_{nk} \}_{(n, k) \in J}$ are the spectral data of $L$.
\end{thm}

In order to prove Theorem~\ref{thm:sd}, consider the data $\{ \la_{nk}^N, \al_{nk}^N \}_{(n, k) \in J}$ defined as follows:
\begin{equation} \label{deflaN}
\la_{nk}^N = \begin{cases}
            \la_{nk}, \quad n \le N, \\
            \tilde \la_{nk}, \quad n > N,
          \end{cases}    \quad
\al_{nk}^N = \begin{cases}
            \al_{nk}, \quad n \le N, \\
            \tilde \al_{nk}, \quad n > N,
          \end{cases}    
\quad N \in \mathbb N.          
\end{equation}

\begin{lem}
The collection $\{ \la_{nk}^N, \al_{nk}^N \}_{(n, k) \in J}$ satisfies conditions (i)-(iii) of Theorem~\ref{thm:nsc} for all sufficiently large $N$. 
\end{lem}

\begin{proof}
Conditions (i)-(ii) are obvious, so we focus on the proof of (iii). We have to show that the sequence
$$
\mathcal Y^N := \{ Y_{nk}^N \}_{(n, k) \in J}, \quad
Y_{nk}^N := \begin{cases}
                Y_{nk}, \quad n \le N, \\
                \tilde Y_{nk}, \quad n > N,
            \end{cases}
$$
is complete in $L_2((0, \pi); \mathbb C^m)$ for each sufficiently large $N$. In view of condition (iii) of Theorem~\ref{thm:nsc}, the sequence $\mathcal X$ is complete, so $\mathcal Y$ is also complete. By virtue of Proposition~\ref{prop:Riesz}, $\mathcal Y$ is a Riesz basis. Consider the sequence
$$
\mathcal Y^{N\bullet} := \{ Y_{nk}^{N\bullet} \}_{(n, k) \in J}, \quad
Y_{nk}^{N\bullet} := \begin{cases}
                Y_{nk}, \quad n \le N, \\
                Y_{nk}^{\bullet}, \quad n > N,
            \end{cases}
$$
where $Y_{nk}^{\bullet}$ is defined similarly to $Y_{nk}$ (see \eqref{defY}), but with $\mathcal E_{nk}$ replaced by $\mathcal E_{nk}^{\bullet} := A_k \mathcal E_{nk}$. It is easy to show that
$$
\lim_{N \to \iy} \sum_{(n, k) \in J, \, n > N} \| Y_{nk} - Y_{nk}^{\bullet} \|_{L_2((0, \pi); \mathbb C^m)}^2 = 0.
$$
Consequently, the sequence $\mathcal Y^{N\bullet}$ is a Riesz basis for sufficiently large $N$, so $\mathcal Y^{N\bullet}$ is complete in $L_2((0, \pi); \mathbb C^m)$ for such $N$. It is easy to check that, for each fixed sufficiently large $n$ and each fixed $k \in \mathcal J$, the vector functions $\{ Y_{ns}^{\bullet} \}_{s \in J_k}$ are linear combinations of $\{ \tilde Y_{ns} \}_{s \in J_k}$. This implies that $\mathcal Y^N$ is also complete in $L_2((0, \pi); \mathbb C^m)$.
\end{proof}

By using $\{ \la_n^N, \al_n^N \}_{(n, k) \in J}$ and the model problem $\tilde L = L(0, T_1, T_2, 0)$, construct the element $\tilde \phi^N(x)$ and the operator $\tilde R^N(x)$ similarly to $\tilde \phi(x)$ and $\tilde R(x)$, respectively. Let $\phi^N(x)$ be the solution of the main equation
\begin{equation} \label{mainN}
\phi^N(x) (\mathcal I + \tilde R^N(x)) = \tilde \phi^N(x), \quad x \in [0, \pi],
\end{equation}
analogous to~\eqref{main}. By virtue of Theorem~\ref{thm:solve}, the solution of~\eqref{mainN} exists and is unique. Obviously, for the matrix sequences $\{ \phi_{nki}^N(x) \}$ and $\{ \tilde \phi_{nki}^N(x) \}$ corresponding to $\phi(x)$ and $\tilde \phi(x)$, respectively, the following relations hold:
$\tilde \phi_{nki}^N(x) = \tilde \phi_{nki}(x)$ for $n \le N$, $\phi_{nk0}^N(x) = \phi_{nk1}^N(x)$, $\tilde \phi_{nk0}^N(x) = \tilde \phi_{nk1}^N(x)$ for $n > N$. Taking these relations into account, similarly to~\eqref{defsi} and~\eqref{defH2}, we define
\begin{align} \label{defsiN}
    \sigma^N(x) := & -2 \sum_{n = 1}^{g(N)} \Biggl( \sum_{\rho_{lsj} \in G_n} (-1)^j \phi_{lsj}^N(x) T_{lsj}^{-1} \al'_{lsj} T_{lsj}^{-1} \tilde \phi_{lsj}(x)  - \frac{1}{2} \bigl( T_1 \hat \al(G_n) T_1 + T_1^{\perp} \hat \al(G_n)  T_1^{\perp}\bigr)
    \Biggr)\\ \label{defH2N}
    H_2^N := & T_2 \sum_{n = 1}^{g(N)} \Biggl( \sum_{\rho_{lsj} \in G_n} (-1)^j \phi_{lsj}^N(\pi) T_{lsj}^{-1} \al'_{lsj} T_{lsj}^{-1} \tilde \phi_{lsj}(\pi) - \bigl( T_1 \hat \al(G_n) T_1 + T_1^{\perp} \hat \al(G_n) T_1^{\perp}\bigr)
    \Biggr) T_2,
\end{align}
where $g(N)$ is such that
$$
\bigcup_{n = 1}^{g(N)} G_n = \{ \rho_{lsj} \colon (l, s) \in J, \, l \le N, \, j = 0, 1 \}.
$$
Here and above, we assume that $N$ is large enough.

Let us show that $\{ \la_{nk}^N, \al_{nk}^N \}_{(n, k) \in J}$ are the spectral data of the problem $L^N := L(\sigma^N, T_1, T_2, H_2^N)$, i.e., prove Theorem~\ref{thm:sd} for $\{ \la_{nk}^N, \al_{nk}^N \}_{(n, k) \in J}$. This special case is much easier for investigation than the general case, since the main equation~\eqref{mainN} in the element-wise form contains a finite sum:
$$
\tilde \phi^N_{nki}(x) = \phi_{nki}^N(x) + \sum_{l, s, j \colon l \le N} (-1)^j \phi^N_{lsj}(x) T_{lsj}^{-1} \al'_{lsj} \tilde D(x, \la_{lsj}, \la^N_{nki}) T_{nki}.
$$
Therefore, one can show that $\tilde R^N(x)$ is twice continuously differentiable with respect to $x \in [0, \pi]$, and so does $\phi^N(x)$ (see the proof of Lemma~1.6.9 from \cite{FY01} for details). Moreover, the sums~\eqref{defsiN} and~\eqref{defH2N} are finite, so we do not need to care of their convergence. Define the matrix functions
\begin{gather} \nonumber
\vv^N(x, \la) := \tilde \vv(x, \la) - \sum_{l, s, j \colon l \le N} (-1)^j \phi^N_{lsj}(x) T_{lsj}^{-1} \al'_{lsj} \tilde D(x, \la_{lsj}, \la), \\ \label{defPhiN}
\Phi^N(x, \la) := \tilde \Phi(x, \la) - \sum_{l, s, j \colon l \le N} (-1)^j \phi^N_{lsj}(x) T_{lsj}^{-1} \al'_{lsj} \tilde E(x, \la_{lsj}, \la), \\ \nonumber
\sigma^N_*(x) := \sigma^N(x) + C^N, \quad H_{2, *}^N := H_2^N - T_2 C^N T_2, \quad
C^N := T_1^{\perp} \sum_{n = 1}^{g(N)} \hat \al(G_n) T_1^{\perp}.
\end{gather}
Calculations yield the following lemma.

\begin{lem} \label{lem:rel}
$\vv^N(., \la) \in C^2([0, \pi]; \mathbb C^{m \times m})$ for each fixed $\la \in \mathbb C$, $\Phi^N(., \la) \in C^2([0, \pi]; \mathbb C^{m \times m})$ for each fixed $\la \ne \la_{nki}$, and $\sigma^N_* \in C^1([0, \pi]; \mathbb C^{m \times m})$.
Moreover, the following relations hold:
\begin{gather*}
    -\frac{d^2}{dx^2}\vv^N(x, \la) + \frac{d}{dx}\sigma^N_*(x) \vv^N(x, \la) = \la \vv^N(x, \la), \quad x \in (0, \pi), \\
    \vv^N(0, \la) = T_1, \quad \frac{d}{dx} \vv^N(0, \la) - \sigma^N_*(0) \vv^N(0, \la) = T_1^{\perp}, \\
    -\frac{d^2}{dx^2}\Phi^N(x, \la) + \frac{d}{dx}\sigma^N_*(x) \Phi^N(x, \la) = \la \Phi^N(x, \la), \quad x \in (0, \pi), \\
    T_1 \left( \frac{d}{dx} \Phi^N(0, \la) - \sigma_*^N(0) \Phi^N(0, \la))\right) - T_1^{\perp} \Phi^N(0, \la) = 0, \\
    T_2 \left( \frac{d}{dx} \Phi^N(\pi, \la) - (\sigma_*^N(\pi) + H_{2, *}^N) \Phi^N(\pi, \la))\right) - T_2^{\perp} \Phi^N(\pi, \la) = 0. 
\end{gather*}
\end{lem}

Lemma~\ref{lem:rel} implies that $\vv^N(x, \la)$ is the $\vv$-type solution and $\Phi^N(x, \la)$ is the Weyl solution of the boundary value problem $L^N_* := L(\sigma^N_*, T_1, T_2, H_{2,*}^N)$. Hence, the Weyl matrix of $L^N_*$ has the form
$$
M^N(\la) := T_1 \Phi^N(0, \la) + T_1^{\perp} \left( \frac{d}{dx}\Phi^N(0, \la) - \sigma_*^N(0) \Phi(0, \la) \right).
$$
Using~\eqref{defPhiN}, we derive
\begin{equation} \label{MN}
M^N(\la) = \tilde M(\la) + \sum_{l, s, j \colon l \le N} \frac{(-1)^j \al'_{lsj}}{\la - \la_{lsj}}.
\end{equation}
Recall that the Weyl matrix $\tilde M(\la)$ has the poles $\{ \tilde \la_{nk} \}_{(n, k) \in J}$ and the corresponding residues $\{ \tilde \al_{nk} \}_{(n, k)}$. Consequently, it follows from~\eqref{deflaN} and \eqref{MN} that the Weyl matrix $M^N(\la)$ has the poles $\{ \la_{nk}^N \}_{(n, k) \in J}$ and the corresponding residues $\{ \al_{nk}^N \}_{(n, k) \in J}$. Thus, $\{ \la_{nk}^N, \al_{nk}^N \}_{(n, k) \in J}$ are the spectral data of the problem $L^N_* = L(\sigma_*^N, T_1, T_2, H_{2,*}^N)$. Since the transform~\eqref{transL} with $H_1^{\perp} = C^N$ does not change the spectral data, we conclude that $\{ \la_{nk}^N, \al_{nk}^N \}_{(n, k) \in J}$ are also the spectral data of $L^N = L(\sigma^N, T_1, T_2, H_2^N)$. Since $\la_{nk}^N \in \mathbb R$ and $\al_{nk}^N = (\al_{nk}^N)^{\dagger}$ for $(n, k) \in J$, one can easily show that the matrices $\sigma^N(x)$ for a.e. $x \in (0, \pi)$ and $H_2^N$ are Hermitian.

The following two lemmas can be proved similarly to Lemmas~5.6 and~5.7 from \cite{Bond-dir}.

\begin{lem} \label{lem:lemN}
$\sigma^N \to \sigma$ in $L_2((0, \pi); \mathbb C^{m \times m})$ and $H_2^N \to H_2$ as $N \to \iy$, where $\sigma$, $H_2$, $\sigma^N$, $H_2^N$ are defined by~\eqref{defsi}, \eqref{defH2}, \eqref{defsiN}, \eqref{defH2N}, respectively.
\end{lem}

\begin{lem} \label{lem:stab}
Suppose that $\sigma$ and $\sigma^N$, $N \in \mathbb N$, are arbitrary Hermitian matrix functions from $L_2((0, \pi); \mathbb C^{m \times m})$ such that $\sigma^N \to \sigma$ in $L_2((0, \pi); \mathbb C^{m \times m})$ as $N \to \iy$ and $H_2$, $H_2^N$, $N \in \mathbb N$, are arbitrary Hermitian matrices from $\mathbb C^{m \times m}$ such that $H^N \to H$ as $N \to \iy$. 
Let $\{ \la_{nk}, \al_{nk} \}_{(n, k) \in J}$ and $\{ \la_{nk}^N, \al_{nk}^N \}_{(n, k) \in J}$ be the spectral data of the problems $L(\sigma, T_1, T_2, H_2)$ and $L(\sigma^N, T_1, T_2, H_2^N)$, respectively.
Then, for each fixed $(n, k) \in J$, 
$$
\lim\limits_{N \to \iy} \la^N_{nk} = \la_{nk}.
$$
Let $\la_{n_1 k_1} = \la_{n_2 k_2} = \dots = \la_{n_r k_r}$ be a group of multiple eigenvalues of $L$, maximal by inclusion. Then
$$
    \lim_{N \to \iy} \sum_{j = 1}^r \al^{N'}_{n_j k_j} = \al_{n_1 k_1}.
$$
\end{lem}

Lemmas~\ref{lem:lemN} and~\ref{lem:stab} together with~\eqref{deflaN} prove Theorem~\ref{thm:sd} for $\{ \la_{nk}, \al_{nk} \}_{(n, k) \in J}$. Theorems~\ref{thm:solve},\ref{thm:sd} and Lemmas~\ref{lem:psi}, \ref{lem:L2} yield the sufficiency part of Theorem~\ref{thm:nsc}. Our proof of sufficiency in Theorem~\ref{thm:nsc} is constructive and provides the following algorithm for solving Inverse Problem~\ref{ip:1}.

\begin{alg} \label{alg:1}
Suppose that the orthogonal projection matrices $T_1$, $T_2$ and the data $\{ \la_{nk}, \al_{nk} \}_{(n, k) \in J}$ satisfying conditions (i)-(iii) of Theorem~\ref{thm:nsc} be given. We have to find $\sigma$ and $H_2$.

\begin{enumerate}
    \item Find $r_k$ and $A_k$ by the formulas
\begin{equation*} 
r_k = \lim_{n \to \iy} (\sqrt{\la_{nk}} - n), \quad 
A_k = \frac{\pi}{2}\lim_{n \to \iy} (T_1 + n^{-1} T_1^{\perp}) \al_n^{(k)} (T_1 + n^{-1} T_1^{\perp}), \quad k = \overline{1, m}.
\end{equation*}
    \item Fix the model problem $\tilde L := L(0, T_1, T_2, 0)$ and find $\{ \tilde \la_{nk}, \tilde \al_{nk} \}_{(n, k) \in J}$, $\{ \tilde \vv(x, \la_{nki}) \}_{(n, k) \in J, \, i = 0, 1}$, by using~\eqref{tvv}, \eqref{tal} and $\tilde \la_{nk} = (\tilde \rho_{nk})^2$.
    \item Find $\tilde D(x, \la_{lsj}, \la_{nki})$ by~\eqref{defD} for $(l, s), (n, k) \in J$, $i, j = 0, 1$.
    \item Divide the values $\{ \rho_{nki} \}$ into the groups $\{ G_n \}_{n = 1}^{\iy}$ according to~\eqref{defG}.
    \item Construct the Banach space $\mathfrak B$, the sequence $\tilde \phi(x) \in \mathfrak B$, and the operator $\tilde R(x) \colon \mathfrak B \to \mathfrak B$ for each fixed $x \in [0, \pi]$ as it is described in Section~3.
    \item Find the solution $\phi(x)$ of the main equation~\eqref{main}.
    \item Using the elements $\{ \phi_{lsj}(x) \}$ of $\phi(x)$, construct $\sigma$ and $H_2$ by formulas~\eqref{defsi} and~\eqref{defH2}, respectively.
\end{enumerate}
\end{alg}

In view of Proposition~\ref{prop:uniq}, the solution constructed by Algorithm~\ref{alg:1} is not the only solution of Inverse Problem~\ref{ip:1}. All the other solutions can be obtained by applying transform~\eqref{transL}.

\noindent Natalia Pavlovna Bondarenko \\
1. Department of Applied Mathematics and Physics, Samara National Research University, \\
Moskovskoye Shosse 34, Samara 443086, Russia, \\
2. Department of Mechanics and Mathematics, Saratov State University, \\
Astrakhanskaya 83, Saratov 410012, Russia, \\
e-mail: {\it BondarenkoNP@info.sgu.ru}

\end{document}